\documentclass[12pt]{amsart}
\usepackage{amsthm}
\usepackage{amsmath}
\usepackage[margin=1.2in]{geometry}
\usepackage{amssymb}
\usepackage{verbatim}
\usepackage{cleveref}
\usepackage{longtable,enumitem,color}
\usepackage{stmaryrd}
\usepackage{graphicx}
\usepackage{url}
\usepackage{amscd}
\usepackage[all]{xy}
\usepackage{enumitem}
	{\providecommand{\noopsort}[1]{} 
\usepackage{chngcntr}\counterwithin{table}{section}

\numberwithin{figure}{section}

\theoremstyle{plain}
\newtheorem{theorem}{Theorem}[section]\newtheorem*{nonumbertheorem}{Theorem}
\newtheorem{corollary}[theorem]{Corollary}
\newtheorem{lemma}[theorem]{Lemma}
\newtheorem{proposition}[theorem]{Proposition}
\newtheorem*{nonumberconjecture}{Conjecture}

\theoremstyle{definition}
\newtheorem{definition}[theorem]{Definition}

\theoremstyle{remark}
\newtheorem{remark}[theorem]{Remark}
\newtheorem{example}[theorem]{Example}
\numberwithin{equation}{section}

\newcommand{\Z}{\mathbb{Z}}\newcommand{\Q}{\mathbb{Q}}

\newcommand{\of}[1]{\left(#1\right)}

\newcommand{\st}{~|~}

\renewcommand{\Z}{{\mathbb{Z}}}

\def\con#1=#2(#3){#1 \equiv #2 \bmod{#3}}

\DeclareMathOperator{\fd}{fd}			

\title[Halperin's conjecture in formal dimensions up to $20$]{Halperin's conjecture in formal dimensions up to $20$}
\author{Lee Kennard}
\address{Department of Mathematics, Syracuse University, Syracuse, NY 13244, U.S.A.}\email{ltkennar@syr.edu}\urladdr{https://ltkennar.expressions.syr.edu/}
\author{Yantao Wu}
\address{Department of Mathematics, Syracuse University, Syracuse, NY 13244, U.S.A.}\email{ywu206@syr.edu}

\begin{document}

\begin{abstract}
A 1976 conjecture of Halperin on positively elliptic spaces was recently confirmed in formal dimensions up to $16$. In this article, we shorten the proof and extend the result up to formal dimension $20$. We work with Meier's algebraic characterization of the conjecture, so the proof is elementary in that it involves only polynomial algebras, ideals, and derivations.
\end{abstract}

\maketitle

\thispagestyle{empty}

\section*{Introduction}
We consider Artinian complete intersection algebras
	\[H^* = \Q[x_1,\ldots,x_k] / (u_1,\ldots,u_k)\]
with a grading concentrated in even degrees. Examples include the rational cohomology of positively elliptic topological spaces, so for simplicity we refer to these algebras as {\it positively elliptic algebras} (see Section \ref{sec:Preliminaries} for definitions).

Positively elliptic spaces play an important role in rational homotopy theory. In fact, they are the subject of a 1976 conjecture of Halperin that is listed as the first of seventeen open problems in \cite[Chapter 39]{FelixHalperinThomas01}. In 1982 Meier \cite{Meier82} proved that this conjecture can be reformulated algebraically as follows (see Section \ref{sec:Preliminaries}):

\begin{nonumberconjecture}[Halperin Conjecture] If $H^*$ is a positively elliptic algebra, then $H^*$ does not admit a non-trivial derivation of negative degree.
\end{nonumberconjecture}

Evidence for this conjecture includes proofs under geometric assumptions such as when $H^*$ is the rational cohomology algebra of a K\"ahler manifold or homogeneous space (see \cite{Blanchard56,Meier83,ShigaTezuka87}). It has also been verified under algebraic assumptions such when $H^*$ is reduced (see \cite{PapadimaPaunescu96}), has at most three generators (see \cite{Lupton90,Chen99}), has relations of large degree (see \cite{ChenYauZuo19}), or has formal dimension at most $16$ (see \cite{AmannKennard20}). In this article, we expand on the latter result by shortening the proof and extending it as follows:

\begin{nonumbertheorem}
Halperin's conjecture holds in formal dimensions at most $20$.
\end{nonumbertheorem}

The proof follows the algebraic setup of \cite{PapadimaPaunescu96, Chen99,ChenYauZuo19} and the inductive strategy in \cite{AmannKennard20} (see Sections \ref{sec:Preliminaries} and \ref{sec:degreetypes}). In addition, we prove two new lemmas, the {\it Large Relations Lemma} and the {\it Top-to-Bottom Lemma} (see Sections \ref{sec:LargeRelations} and \ref{sec:Top-to-Bottom}). These serve to efficiently rule out all cases except for two in formal dimension $20$. The proof is completed in Section \ref{sec:Proof}.

\bigskip\noindent{\bf Acknowledgements:} This first author was supported by NSF Grant DMS-2005280, and the second was supported by a grant from the Syracuse Office of Undergraduate Research and Creative Engagement (SOURCE) at Syracuse University. Both authors would like to thank Manuel Amann and Claudia Miller for helpful discussions while preparing this paper.

\bigskip\section{Preliminaries}\label{sec:Preliminaries}\medskip
Let $A = \Q[x_1,\ldots,x_k]$ denote the polynomial ring on $k$ variables. Assume moreover that each $x_i$ has a positive, even degree assigned to it denoted by $|x_i|$. This induces a {\it graded algebra} structure on $A = \bigoplus_{n\geq 0} A^n$ where the subspace $A^n$ is spanned by monomials $x_1^{a_1}\cdots x_k^{a_k}$ satisfying $a_1|x_1| + \ldots + a_k |x_k| = n$.

Next let $I = (u_1,\ldots,u_k)$ denote the ideal generated by homogeneous polynomials $u_i \in A^{|u_i|}$, where $|u_i|$ denotes the degree of $u_i$. Recall that the $u_i$ form a {\it regular sequence} if $u_1 \in A$ is non-zero and if the image of $u_i$ in $A/(u_1,\ldots,u_{i-1})$ is not a zero divisor for all $2 \leq i \leq k$. 

\begin{definition}\label{def:Q}
The quotient $H^* = \Q[x_1,\ldots,x_k]/(u_1,\ldots,u_k)$ of a graded polynomial ring on generators with even degrees by an ideal generated by a regular sequence $u_1,\ldots,u_k$ of homogeneous polynomials is called an {\it positively elliptic algebra}.
\end{definition}

\begin{remark}
Algebras presented as in Definition \ref{def:Q} are also called graded (or weighted or quasi-homogeneous) Artinian complete intersection algebras with grading concentrated in even degrees. The condition that all generators have even degree is not required to state the Halperin conjecture. However we require it here to maintain the connection to rational homotopy theory, and this is the motivation for the definition (see \cite{Lupton98, AmannKapovitch12,AmannKollross-pre}). For example, the formal dimension defined below and appearing in our theorem equals the dimension of the underlying rationally elliptic manifold $M$ when $H^*$ is the rational cohomology of $M$.
\end{remark}

Singly generated algebras of the form $H^* = \Q[x_1]/(x_1^\alpha)$ are always positively elliptic. Doubly generated algebras can be positively elliptic or not, as can be seen from the examples $\Q[x_1,x_2]/(x_1^2 - x_2^2, x_1x_2)$ or $\Q[x_1,x_2]/(x_1^2, x_1x_2)$. Note in the latter case that the image of $x_1x_2$ in $\Q[x_1,x_2]/(x_1^2)$ is a zero divisor, so the ideal is not generated by a regular sequence.

To better understand positively elliptic algebras, we review a few well known facts. First, since the $u_i$ are homogeneous, the grading on $A$ descends to $H^* = \bigoplus_{n\geq 0} H^n$. Note that $H^n = 0$ for odd $n$ since the degrees $x_i$ are even.

Second, the quotient $H^* = \Q[x_1,\ldots,x_k]/(u_1,\ldots,u_k)$ by an arbitrary sequence of homogeneous polynomials $u_i$ of positive degree has {\it finite dimension} $\dim H^* = \sum \dim H^n$ (or Krull dimension zero) if and only if the $u_i$ form a regular sequence. Therefore another way to define a positively elliptic algebra is by requiring that $\dim H^* < \infty$.

Third, $H^*$ satisfies Poincar\'e duality (see \cite[Section8]{Halperin77}). This means that there exists $n \geq 0$ such that $H^i = 0$ for $i > n$ and $H^n \cong \Q$ and that the product map $H^i \times H^{n-i} \to H^n \cong \Q$ is a non-degenerate bilinear pairing for all $0 \leq i \leq n$. The integer $n$ is called the {\it formal dimension} (or socle degree) and is denoted by $\fd H^*$ (cf. \cite{CattalaniMilivojevic20}).

Positively elliptic algebras arise as the rational cohomology algebras of rationally elliptic topological spaces $F$ with positive Euler characteristic. In the literature, such spaces are said to be $F_0$ or positively elliptic, and they were conjectured by Halperin in 1976 to satisfy the following: For a fibration with fiber $F$, the Serre spectral sequence of that fibration degenerates (see \cite[Chapter 39]{FelixHalperinThomas01}).

In 1982, Meier \cite{Meier82} proved that Halperin's conjecture can be reformulated entirely algebraically in terms of negative degree derivations as stated in the introduction. Indeed, any non-trivial differential $d_{r+1}$ in the spectral sequence induces a derivation $\delta$ on $H^*(F)$ of (negative) degree $-r$. Recall that a derivation on $H^*$ is a linear map $\delta:H^* \to H^*$ that increases degree by some (possibly non-positive) integer $|\delta| \in \Z$ and behaves on products of homogeneous elements as follows:
	\[\delta(xy) = \delta(x) y + (-1)^{|\delta||x|} x \delta(y).\]

\begin{example}\label{exa:2gen}
The graded algebra $H^* = \Q[x_1,x_2]/(x_1^2 - \lambda x_2^2, x_1x_2)$ with $|x_1| = |x_2| = 2$ and $\lambda \in \Q\setminus\{0\}$ is a positively elliptic algebra and admits a non-trivial derivation $\delta$ of degree 2. Indeed, if we define $\delta(x_1) = x_1^2$ and $\delta(x_2) = 0$ and extend the definition by linearity and the Leibniz rule, we obtain a well defined derivation on $\Q[x_1,x_2]$. In addition, $\delta(x_1^2 - \lambda x_2^2)$ and $\delta(x_1x_2)$ are in the ideal $(x_1^2 - \lambda x_2^2, x_1x_2)$, so $\delta$ descends to a non-trivial derivation on $H^*$.
\end{example}

This example also demonstrates the way in which we work with derivations on $H^*$. They correspond exactly to derivations on $\Q[x_1,\ldots,x_k]$ that map the ideal $(u_1,\ldots,u_k)$ into itself. We note again here that we use the same notation for the generators $x_i$ in $\Q[x_1,\ldots,x_k]$ and in the quotient $H^*$.

We now restate Meier's reformulation of Halperin's conjecture from the introduction here for easy reference:

\begin{nonumberconjecture}[Halperin Conjecture] Positively elliptic algebras do not admit non-trivial derivations of negative degree.
\end{nonumberconjecture}

We close this preliminary section with two basic results (see \cite[Lemmas 11.1 and 11.3]{AmannKennard20}).

\begin{lemma}[Land in Zero Lemma]\label{lem:land}
Let $\delta$ be a derivation of negative degree on $H^*$. If $x\in H^{i}$ for some $i>0$ such that $\delta(x)\in H^0$, then $\delta(x)=0$. 
\end{lemma}

\begin{lemma}[$k-1$ Lemma]\label{lem:kmo1}
If $\delta$ is a derivation of negative degree on $H^*$ such that $\delta(x_i)=0$ for $k-1$ of the $k$ generators $x_i$, then $\delta = 0$.
\end{lemma}

These lemmas imply Thomas' result that the Halperin Conjecture holds when $H^*$ is generated by at most two elements (see \cite{Thomas81}). 

\bigskip\section{Degree types and splittings}\label{sec:degreetypes}\medskip

Given a positively elliptic algebra $H^* = \Q[x_1,\ldots,x_k]/(u_1,\ldots,u_k)$, the {\it degree type} of $H^*$ is the sequence of even, positive integers denoted by 
	\[(|x_1|,\ldots,|x_k|;|u_1|,\ldots,|u_k|).\] 
As Example \ref{exa:2gen} shows, the degree type $(2, 2; 4, 4)$ can be realized in infinitely many ways, even up to isomorphism. This is a general feature. Nevertheless, it is helpful to sort positively elliptic algebras according to their degree types. In this section, we summarize previous work on degree types as they relate to Halperin's conjecture. In addition, we define pure models, formal dimension, and splittings. The first basic result is the following (see \cite[Section 32]{FelixHalperinThomas01}):

\begin{theorem}[Pure model]\label{thm:PureModel} Given a non-zero positively elliptic algebra $H^*$, there exist variables $x_i$, positive and even degrees $|x_i|$, and homogeneous polynomials 
	\[u_i \in \Q^{\geq 2}[x_1,\ldots,x_k] = \mathrm{span}\{x_1^{a_1}\cdots x_k^{a_k} \st a_1 + \ldots + a_k \geq 2\}\]
such that $H^* \cong \Q[x_1,\ldots,x_k]/(u_1,\ldots,u_k)$. Moreover these choices can be made to satisfy all of the following:
	\begin{enumerate}
	\item $|x_1| \leq \ldots \leq |x_k|$.
	\item $|u_1| \leq \ldots \leq |u_k|$.
	\item $|u_i| \geq 2|x_i|$ for all $1 \leq i \leq k$.
	\end{enumerate}
In addition, the formal dimension satisfies
	\[\fd H^* = \sum_{k=1}^n \of{|u_i| - |x_i|}.\]
\end{theorem}

Such a presentation of $H^*$ is called a {\it pure model}, and we assume from now on that our presentations of positively elliptic algebras are pure models.

\begin{proof}
By definition, there is some presentation of $H^* = \Q[x_1,\ldots,x_k]/(u_1,\ldots,u_k)$ with $k \geq 1$. We may assume that $k$ is minimal.

Clearly the $u_j$ do not have constant terms, since otherwise the ideal $I = (u_1,\ldots,u_k)$ is the entire polynomial algebra. Moreover, if some $u_j$ has a linear term equal to a multiple of $x_i$, then the automorphism of $\Q[x_1,\ldots,x_k]$ that replaces $x_i$ by $u_j$ is an isomorphism. Taking the quotient by $I$ gives rise to a presentation of $H^*$ on $k-1$ generators. This contradicts the minimality of $k$, so we have that each $u_j \in \Q^{\geq 2}[x_1,\ldots,x_k]$.

Next, we may relabel the generators and relations so that $|x_1| \leq \ldots \leq |x_k|$ and $|u_1| \leq \ldots \leq |u_k|$. The final condition that $|u_i| \geq 2|x_i|$ for all $i$ follows by the result of Friedlander and Halperin below (Theorem \ref{thm:SAC}). Indeed, this result implies that some relation (and hence $u_k$) has degree at least twice $|x_k|$, that at least two relations (and hence $u_{k-1}$ and $u_k$) have degree at least twice $|x_{k-1}|$, and so on. 

The last claim holds by the well known fact that $H^*$ satisfies Poincar\'e duality and that one choice of fundamental class is given by the Jacobian $\det\of{{\partial u_i}/{\partial x_j}}$ (see, for example, the remarks following Theorem B in \cite{ShigaTezuka87}).
\end{proof}

\begin{example}
The positively elliptic algebra $H^* = \Q[x_1,x_2]/(x_1^2 - x_2, x_2^3)$ with $|x_1| = 2$ and $|x_2| = 4$ can be more efficiently presented as $H^* \cong \Q[x]/(x^6)$. Indeed, an isomorphism is given by mapping $x_1 \mapsto x$ and $x_2 \mapsto x^2$. 
\end{example}

A consequence of Theorem \ref{thm:PureModel} is that any given formal dimension only allows for finitely many degree types. Indeed, 
	\[\fd H^* = \sum_{i=1}^k |u_i| - |x_i| \geq \sum_{i=1}^k |x_i| \geq 2k,\]
so $k \leq \frac 1 2 \fd H^*$, the possible degrees $|x_i|$ are similarly bounded, and therefore the possibilities for the $|u_i|$ are finite. 

A further restriction on the degree types is the following result due to Friedlander and Halperin (see \cite[Corollary 1.10]{FriedlanderHalperin79} or \cite[Proposition 32.9]{FelixHalperinThomas01}):

\begin{theorem}[Characterization of degree types]\label{thm:SAC} A sequence \[(A_1,\ldots,A_k; B_1,\ldots B_k)\] of positive, even integers arises as the degree type of some positively elliptic algebra if and only if the following holds: For all $1 \leq l \leq k$ and $1 \leq i_1 < \ldots < i_l \leq k$, there exist $1 \leq j_1 < \ldots < j_l \leq k$ such that $B_{j_1},\ldots,B_{j_l}$ can be expressed as linear combinations of the form $\lambda_1 A_{i_1} + \ldots + \lambda_l A_{i_l}$ with non-negative integer coefficients satisfying $\lambda_1 + \ldots + \lambda_l \geq 2$.
\end{theorem}

To illustrate, the degree type $(A_1,A_2; B_1, B_2) = (2,4;4,10)$ does not satisfy this condition since $A_2 = 4$ does not properly divide any of the $B_j$. Similarly, the degree type $(2,2,4,4; 4,6,8,10)$ does not satisfy the condition and therefore does not arise as the degree type of a positively elliptic algebra. 

\begin{definition}\label{def:SAC}
A sequence $(A_1,\ldots,A_k; B_1, \ldots, B_k)$ as in Theorem \ref{thm:SAC} satisfies the condition $SAC(A_{i_1}, \ldots, A_{i_l})$ if there exist $B_{j_1},\ldots,B_{j_l}$ as in the theorem. 
\end{definition}

In \cite{FriedlanderHalperin79}, the condition that $SAC(A_{i_1},\ldots,A_{i_l})$ holds for all possible subsequences of indices $1 \leq i_1 < \ldots < i_l \leq k$ is called the Strong Algebraic Condition (SAC). The examples after the theorem fail the SAC(4) and the SAC(4,4), respectively, and therefore both fail the SAC.

We close this section with a discussion of positively elliptic algebras $H^*$ that decompose in a certain sense.

\begin{definition}[Split positively elliptic algebras]
A positively elliptic algebra $H^*$ {\it splits} if there exist non-zero positively elliptic algebras $K^*$ and $Q^*$ such that the sequence
	\[0 \to K^* \to H^* \to Q^* \to 0\]
is exact.
\end{definition}

An example of a splitting is if $H^*$ has a presentation as a pure model
	\[H^* \cong \Q[x_1,\ldots,x_k]/(u_1,\ldots, u_k)\]
such that, for some $0 < l < k$, the polynomials $u_1,\ldots,u_l$ only depend on $x_1,\ldots,x_l$. Indeed, in this case, all of the following are easy to verify:
	\begin{enumerate}
	\item The subalgebra $K^* = \Q[x_1,\ldots,x_l]/(u_1,\ldots,u_l)$ is a positively elliptic algebra.
	\item The quotient algebra $Q^* = H^*/K^* \cong \Q[\bar x_{l+1},\ldots,\bar x_k]/(\bar u_{l+1}, \ldots, \bar u_k)$ is a positively elliptic algebra, where the bars denote images under the projection map.
	\item The formal dimensions of $K^*$ and $Q^*$ are positive and sum to $\fd H^*$.
	\end{enumerate}
In the proof of the Halperin conjecture up to dimension $20$, we will proceed by induction over the formal dimension. In particular, the following is an important tool for dealing with the split case (see \cite[Theorem 1]{Markl90}):

\begin{theorem}[Markl's theorem]
Let $H^*$ be a positively elliptic algebra with a non-zero derivation of negative degree. If $H^*$ splits as 
	\[0 \to K^* \to H^* \to Q^* \to 0,\] 
then $K^*$ or $Q^*$ also has a non-zero derivation of negative degree.
\end{theorem}

In our situation, where the splitting of
	\[H^* \cong \Q[x_1,\ldots,x_k]/(u_1,\ldots,u_k)\]
is adapted to this choice of pure model in the sense that the polynomials $u_1,\ldots,u_l$ only depend on $x_1,\ldots,x_l$ for some $0 < l < k$, Markl's proof takes on a somewhat simpler form, so we include it here:

\begin{proof}
Assume that $H^* = \Q[x_1,\ldots,x_k]/(u_1,\ldots,u_k)$ is a pure model for a positively elliptic algebra $H^*$ with the property that \[u_1,\ldots,u_l \in \Q[x_1,\ldots,x_l]\] for some $0 < l < k$. Let
	\[K^* = \Q[x_1,\ldots,x_l]/(u_1,\ldots,u_l).\]

Let $\delta$ be a derivation of negative degree on $H^*$. We assume that both $K^*$ and $H^*/K^*$ do not admit non-zero derivations of negative degree, and we use this to prove that $\delta = 0$.

First, since the degrees of the $x_i$ are increasing, the derivation $\delta$ restricts to a derivation on $K^*$. By the assumption on $K^*$, we have
	\[\delta(x_1) = \ldots = \delta(x_l) = 0.\]

Next, fix any vector space basis $\{\xi_\alpha\}$ for $K^*$ consisting of monic polynomials $\xi_\alpha$ in the variables $x_1,\ldots,x_l$. For each $y \in H^*$, there exist polynomials $\delta_\alpha(y)$ in the variables $x_{l+1},\ldots,x_k$ such that
	\[\delta(y) = \sum_\alpha \xi_\alpha \delta_\alpha(y).\]
We claim that each of the maps
	\begin{eqnarray*}
	\bar\delta_\alpha:H^*/K^* &\to& H^*/K^*\\
				\bar y &\mapsto& \overline{\delta_\alpha(y)}
	\end{eqnarray*}
is a well defined derivation of negative degree on $H^*/K^*$ and hence is equal to zero by assumption. If this claim is true, it is straightforward to see that 
	\[\delta(x_j) = 0 \mathrm{~for~all~} l < j \leq k\]
since otherwise some $\delta_\alpha(x_j)$ is a non-zero polynomial only depending on $x_{l+1},\ldots,x_k$, which would imply that
	\[\bar\delta_\alpha(\bar x_j) = \overline{\delta_\alpha(x_j)} \neq 0\]
in $H^*/K^*$, a contradiction.

To prove the claim, we first consider the issue of being well defined. It suffices to show that $\delta_\alpha$ maps the ideal $(x_1,\ldots,x_i)$ into itself. Fix such an element in this ideal, and express it as
	\[z = \sum_{i=1}^l x_i z_i.\]
Since $\delta$ is a derivation on $H^*$ that vanishes on $x_1,\ldots,x_l$, this equation implies the following upon applying $\delta$:
	\[\sum_{\alpha} \xi_\alpha \delta_\alpha(z) 
	= \sum_{i=1}^l x_i  \sum_\alpha \xi_\alpha \delta_\alpha(z_i).\]
For any $\alpha$, define $\delta_{\alpha,i}$ to be zero if the monomial $\xi_\alpha$ has no $x_i$ term and otherwise to be $\delta_\beta$ where $\beta$ is the index such that $x_i \xi_\beta = \xi_\alpha$. Hence
	\[\delta_\alpha(z) = \sum_{i=1}^{l} \delta_{\alpha,i}(z_i).\]

We are ready to prove by induction that $\bar\delta_\alpha= 0$ for any $\alpha$. For the base case, let $\alpha_0$ denote the index corresponding to the constant polynomial $\xi_{\alpha_0} = 1$. The calculation above shows that $\bar\delta_{\alpha_0}$ maps the ideal $(x_1,\ldots,x_l)$ to zero, so $\bar \delta_{\alpha_0}$ is a well defined function on $H^*/K^*$. Similar, but simpler, calculations use the fact that $\delta$ is linear and satisfies the Leibniz rule imply that these properties also hold for $\delta_{\alpha_0}$. Since the degree of $\delta_{\alpha_0}$ equals that of $\delta$, it must vanish by our assumption on $H^*/K^*$.

For the inductive step, fix any other index $\alpha$ and assume that $\bar \delta_\beta = 0$ for all $\beta < \alpha$. The computation of $\delta_\alpha(z)$ above, together with the induction hypothesis, implies that $\delta_\alpha$ descends to a well defined function on $H^*/K^*$. Again it follows that $\bar\delta_\alpha$ is a derivation on $H^*/K^*$ with degree $|\delta| - |\xi_\alpha| < |\delta| < 0$. Hence the assumption on $H^*/K^*$ implies that $\bar\delta_\alpha = 0$. This completes the proof of the claim and hence of Markl's theorem.
\end{proof}

For the inductive proof of our main theorem, Markl's theorem implies there is nothing to do in the case where $H^*$ splits. Therefore it is useful to have conditions that imply the existence of splittings. The following is an example of this (see \cite{AmannKennard20}):

\begin{lemma}[Degree Inequality]\label{lem:degreeInequality}
Let $H^* = \Q[x_1,\ldots,x_k]/(u_1,\ldots,u_k)$ be a positively elliptic algebra that does not split. The following hold:
	\begin{enumerate}
	\item If $i < k$, then $|u_i| \geq |x_1| + |x_{i+1}|$.
	\item If $\delta(x_2) = \lambda x_1^\alpha \neq 0$ for some $\lambda \in \Q$, where $\delta$ is a derivation on $H^*$ with negative degree, then $|u_1| \geq |x_1| + |x_3|$.
	\end{enumerate}
\end{lemma}

\begin{proof}
The first claim is a restatement of \cite[Lemma 11.4]{AmannKennard20}. It follows easily, since $|u_i| < |x_1| + |x_{i+1}|$ for some $i$ implies that $u_j \in \Q^{\geq 2}[x_1,\ldots,x_k]$ is a polynomial in $x_1,\ldots,x_i$ for degree reasons for all $1 \leq j \leq i$. Hence $x_1,\ldots,x_i$ generate a non-trivial subalgebra, a contradiction.

The second claim is implicit in the proof of \cite[Lemma 11.5]{AmannKennard20}. Suppose that $\delta(x_2) = \lambda x_1^\alpha \neq 0$ for some $\lambda \in \Q$ and $\alpha \geq 1$. Suppose for the purpose of contradiction that $|u_1| < |x_1| + |x_3|$. As in the previous paragraph, we conclude that $u_1$ is a polynomial in $x_1$ and $x_2$. Write $u_1 = \sum_{i=0}^r p_i(x_1) x_2^i$. Since $\delta(u_1)$ is in the ideal $(u_1,\ldots,u_k)$ and at the same time has degree less than any of the $u_i$, we have $\delta(u_1) = 0$. On the other hand, $\delta(u_1) = \sum_{i=1}^r \lambda i p_i(x_1) x_1^\alpha x_2^{i-1}$, so $p_i(x_1) = 0$ for all $i \geq 1$. Hence $u_1 = p_0(x_1)$, $x_1$ generates a non-trivial subalgebra of $H^*$, and we have a contradiction.
\end{proof}

\bigskip\section{The Large Relations Lemma}\label{sec:LargeRelations}\medskip
In this section, we prove the Large Relations Lemma and use it to prove Corollary \ref{cor:8}. This is the main step in proving the Halperin Conjecture when the degrees of the two largest generators satisfy $|x_{k-1}| + |x_k| \leq 8$.


\begin{lemma}[Large Relations Lemma]\label{lem:LargeRelations}
Let $H^* \cong \Q[x_1,\ldots,x_k]/(u_1,\ldots,u_k)$ be a positively elliptic algebra that does not split. Assume that $H^*$ admits a derivation $\delta$ of negative degree such that the map $\delta:H^4 \to H^2$ has rank $m \geq 1$.

Let $g_i$ denote the number of generators with degree $i$, and let $r_j$ denote the number of relations with degree $j$. The following hold:
	\begin{enumerate}
	\item If $g_6 + g_{10} + g_{14} + \ldots = 0$, then 
		\[r_{12} + r_{16} + \ldots \geq \of{k-g_2-g_4} + \max(1, m-r_4).\]
	\item If $g_6 + g_{10} + g_{14} + \ldots \geq 1$ and $\delta^2(H^6) = 0$, then 
		\[r_{10} + r_{12} + \ldots \geq \of{k-g_2-g_4} + \max(1, m-r_4).\]
	\end{enumerate}
In particular, $|u_k| \geq 12$ in the first case and $|u_{k-1}| \geq 10$ in the second.
\end{lemma}

\begin{proof}
We prove the claims simultaneously. By the Land in Zero Lemma, we may assume that $\delta(x_i) = 0$ for $1 \leq i \leq g_2$. In addition, we may change basis so that
	\[\delta(x_{g_2+i}) = \left\{\begin{array}{rcl} x_{i} &\mathrm{if}& 1 \leq i \leq m\\
					0	&\mathrm{if}&	m < i \leq g_4\end{array}\right.\]
Finally, if $x_h$ is a generator in degree six, then the condition $\delta^2(H^6) = 0$ implies that $\delta(x_h)$ has no $x_{g_2+i}$ term with $1 \leq i \leq m$.

Let $\{u_j | j \in J\}$ denote the relations with degree $8$. Write each of these as
	\[u_j = p_j(x_{g_2+1},\ldots,x_{g_2+m}) + r_j\]
where $p_j$ is a quadratic polynomial and where $r_j$ is in the ideal 
	\[I_0 = (x_1,\ldots,x_{g_2}) + (x_{g_2 + m+1}, \ldots, x_{g_2+g_4}).\]
Fix $J' \subseteq J$ such that $\{p_j \st j \in J'\}$ is a basis for the span of $\{p_j \st j \in J\}$. 

Suppose for a moment that $|J'| \geq m$. Since $I_0 \cap H^4 \subseteq \ker(\delta)$, $H^6 \subseteq \ker(\delta^2)$, and $|r_j| = 8$, it follows that $\delta^2(r_j) = 0$ for all $j \in J'$. Therefore
	\[\delta^2(u_j) = 2 p_j(x_1,\ldots,x_m).\]
for $j \in J'$. Now $\delta^2(u_j)$ has degree four and lies in the ideal generated by $u_1,\ldots,u_k$. Hence $p_j(x_1,\ldots,x_m)$ lies in the $r_4$-dimensional span of $\{u_i \st |u_i| = 4\}$. Since the polynomials $p_j$ with $j \in J'$ are linearly independent, we may perform a change of basis on the degree-four relations $u_i$ such that $u_1,\ldots,u_m$ are equal to the polynomials $p_j(x_1,\ldots,x_m)$ for $j \in J'$. But this implies that $x_1,\ldots,x_m$ generate a subalgebra $K^*$ that induces a splitting of $H^*$, a contradiction to the assumptions in the lemma.

We may assume now that $|J'| \leq m - 1$. By the argument in the previous paragraph, $|J'| \leq \min(m-1, r_4)$ by choice of $J'$. We can perform a change of basis on the $u_j$ for $j \in J$ so that $p_j = 0$ for $j \in J \setminus J'$. 

To finish the proof of Claim 1, consider the ideal
	\[I = I_0 + \of{\{u_j \st j \in J'\}} + \of{\{u_j \st |u_j| \in\{12,16,\ldots\}}.\]
If a relation $u_i$ has degree less than eight or not divisible by four, then it lies in $I_0$ for degree reasons since there are no generators in degrees $6$, $10$, etc. If $|u_i| = 8$, then it lies in $I_0 + \of{\{u_j \st j \in J'\}}$ by choice of $J'$. Finally it is clear that $u_i \in I$ for all other relations $u_i$. Hence $H^*$ projects onto $\Q[x_1,\ldots,x_k]/I$. Since $H^*$ is finite-dimensional, $I$ must have at least $k$ generators. Therefore
	\[(g_2 + g_4 - m) + \min(m-1,r_4) + (r_{12} + r_{16} + \ldots) \geq k,\]
which implies the desired bound in Claim 1.

To finish the proof of Claim 2, we use a similar argument with $I$ replaced by
	\[I = I_0 + \of{\{u_j \st j \in J'\}} + \of{\{u_j \st |u_j| \in\{10,12,\ldots\}}.\]
It is clear that relations of degree four or degree eight or larger lie in $I$. Relations of degree six are also in $I_0$ and hence in $I$ because they are polynomials in $\Q^{\geq 2}[x_1,\ldots,x_k]$ (see Theorem \ref{thm:PureModel}). Hence again all relations are in $I$, and the claim follows as before.
\end{proof}

Next, we apply Lemma \ref{lem:LargeRelations} to quickly prove our main theorem when $|x_{k-1}| + |x_k| \leq 8$ in all but three exceptional cases. 

\begin{proposition}\label{pro:8} Let $H^*$ be a positively elliptic algebra that does not split. If $H^*$ admits a non-zero derivation of negative degree and $|x_{k-1}| + |x_k| \leq 8$, then either $\fd H^* > 20$ or the degree type is one of the following:
	\[(2,2,4,4; 4,6,8,12),  (2,2,4,4; 4,8,8,12),  \mathrm{~or~} (2,2,2,4,4; 4,4,6,8,12),\]
\end{proposition}

\begin{proof}
By the Land in Zero and $k-1$ Lemmas, we may assume that $|x_{k-1}| = |x_k| = 4$ and that $\delta:H^4 \to H^2$ has rank $m \geq 2$. Hence the Lemma \ref{lem:LargeRelations} implies that $|u_k| \geq 12$. This forces the formal dimension to be large, so we estimate it now.

Let $g_4 \geq m$ denote the number of generators of degree four. Using the formula for the formal dimension in Theorem \ref{thm:PureModel} and the Degree Inequality (Lemma \ref{lem:degreeInequality}), we have
	\begin{eqnarray*}
	\fd H^* 
	&\geq& \sum_{i=1}^{k-g_4} \of{|x_1| + |x_{i+1}| - |x_i|} + \sum_{i=k-g_4+1}^{k-1} {|x_i|} ~+~ \of{12 - |x_k|}\\
	&=& 2k + 2g_4 + 6.
	\end{eqnarray*}

If $g_4 \geq 3$, then $k \geq m + g_4 \geq 5$ and hence $\fd H^* > 20$. This is what we wish to show, so we may assume that $g_4 = m = 2$. The degree type is of the form
	\[(2,\ldots, 2, 4, 4; B_1, \ldots, B_k)\]
with $B_i \geq 2|x_i| = 4$ for $1 \leq i \leq k-3$, $B_{k-2} \geq |x_1| + |x_{k-1}| = 6$, $B_{k-1} \geq 2|x_{k-1}| = 8$, and $B_k \geq 12$. Going back to the estimate on $\fd H^*$, we see that $k \in \{4,5\}$.

If $k = 4$, then $\fd H^* = \sum_{i=1}^4 B_i - 12 \geq 18$. Since we may assume that $\fd H^* \leq 20$, it follows either that we have equality in all four of the lower bounds on the $B_i$ or that we have equality in three of the four bounds and we are off by two in the fourth. This gives rise to five possibilities for the degree type. Two of these are ruled out by the SAC(4,4) condition, one is ruled out by the bound $r_{12} + r_{16} + \ldots \geq m - r_4$ from Lemma \ref{lem:LargeRelations}, and the remaining two appear in the conclusion of the proposition.

If instead $k = 5$, then we estimate as above: $\fd H^* = \sum_{i=1}^5 B_i - 14 \geq 20$. Hence equality holds in all five of the lower bounds on the $B_i$, and we find that the degree type is the last one shown in statement of the proposition.
\end{proof}

\begin{proposition}\label{pro:8part2}
Let $H^*$ be a positively elliptic algebra that does not split. If the degree type is
	\[(2,2,4,4; 4,6,8,12),  (2,2,4,4; 4,8,8,12),  \mathrm{~or~} (2,2,2,4,4; 4,4,6,8,12),\]
then there does not exist a non-zero derivation with negative degree.
\end{proposition}

\begin{proof}
We keep the notation from Lemma \ref{lem:LargeRelations}, with a slight modification. We may assume 
	\[\delta(x_{k-1}) = x_1 \mathrm{~and~} \delta(x_k) = x_2\]
and that $\delta(x_i) = 0$ for $1 \leq i \leq k-2$. In addition, after possibly swapping the two degree eight relations in the second case, we may assume that
	\[u_{k-1} = p_{k-1}(x_{k-1}, x_k) + r_{k-1}\]
with $p_{k-1} \neq 0$ and $r_{k-1} \in (x_1,\ldots,x_{k-2})$. Indeed, if $p_{k-1} = 0$ (and $p_{k-2} = 0$ in the second case), then $H^*$ admits a quotient map onto $\Q[x_1,\ldots,x_k]/(x_1,\ldots,x_{k-2},u_k)$, a contradiction to finite-dimensionality.

Applying $\delta^2$ as in the proof of Lemma \ref{lem:LargeRelations}, we find that 
	\[p_{k-1}(x_1,x_2) = u_1\]
after possibly changing basis in the degree four relations. In addition, in the case where $u_{k-2}$ also has degree eight, we find that $p_{k-2}$ is a multiple of $u_1$, where $u_{k-2} = p_{k-2}(x_{k-1}, x_k) + r_{k-2}$ and $r_{k-2} \in (x_1,\ldots,x_{k-2})$. In this case, we can replace $u_{k-2}$ by $u_{k-2} - \mu u_{k-1}$ for some $\mu \in \Q$ so that $p_{k-2} = 0$. In any case, we have shown that
	\[u_1,\ldots,u_{k-2} \in (x_1,\ldots,x_{k-2}).\]

We now extend the argument from Lemma \ref{lem:LargeRelations} by considering the degree $12$ relation $u_k$. Write
	\[u_k = p_k(x_{k-1}, x_k) + r_k\]
for some cubic polynomial $p_k$ and some $r_k \in (x_1,\ldots,x_{k-2})$. For degree reasons, we have that $\delta^3(r_k) = 0$ and hence that
	\[6 p_k(x_1,x_2) = \delta^3(u_k) \in (u_1,\ldots,u_k).\]
Note that $p_k(x_1,x_2)$ has degree six and can be expressed as
	\[p_k(x_1,x_2) = \sum_{i=1}^k h_i u_i\]
where $h_i \in \Q[x_1,\ldots,x_k]$ is a linear polynomial in the first $k-2$ variables if $|u_i| = 4$, where $h_i \in \Q$ if $|u_i| = 6$, and where $h_i = 0$ if $|u_i| \geq 8$.

We further claim that $h_i=0$ when $|u_i|=6$. Indeed, otherwise we can replace $u_i$ by the expression $\sum h_i u_i$ so that $u_i = p_k(x_1,x_2)$. For the degree types under consideration, this implies that $x_1,x_2,\dots, x_{k-2}$ generate a subalgebra $K^*$ that induces a splitting of $H^*$, a contradiction. We may therefore assume that $p_k(x_1,x_2) = h_1 u_1$ in the first two cases and that $p_k(x_1,x_2) = h_1 u_1 + h_2 u_2$ in the third.

To derive a contradiction in the first two cases (where $k = 4$), recall that $u_1 = p_{3}(x_1,x_2)$ and hence that $p_4(x_3,x_4)$ is in the ideal
	\[I = (x_1,x_2, p_3(x_3,x_4)).\]
For degree reasons, it follows that $I$ contains all four of the $u_i$ and hence that there exists a projection of $H^*$ onto $\Q[x_1,\ldots,x_4]/I$. Since the latter space has infinite dimension, this is a contradiction. 

To derive a contradiction in the last case (where $k = 5$), we consider the expression
	\[p_5(x_1,x_2) = h_1 u_1 + h_2 u_2.\]
Write $h_i = l_i(x_1,x_2) + k_i x_3$ for some linear polynomials $l_i$ and some $k_i\in\Q$, and write $u_2 = u_{2,0}(x_1,x_2) + x_3 u_{2,1}(x_1,x_2,x_3)$. We break the proof into cases.

\begin{itemize}
\item Suppose $u_{2,1} = 0$. This implies that $u_2$ is a polynomial in $x_1$ and $x_2$. Since $u_1 = p_4(x_1,x_2)$ as well, we see that $x_1$ and $x_2$ generate a subalgebra that induces a splitting of $H^*$, a contradiction. 

\item Suppose $h_2 = 0$. This implies that $u_1 = p_{4}$ divides $p_5$. Hence the ideal 
	\[I = (x_1,x_2,x_3,p_{4}(x_4,x_5))\]
contains all of the $u_j$, a contradiction to finite-dimensionality of $H^*$. 

\item Suppose instead that $u_{2,1} \neq 0$ and that $h_2 \neq 0$. Comparing coefficients of $x_3^2$ and $x_3^3$ in the above equation, we see that $h_2 = l_2 \neq 0$. Similarly, comparing coefficients of $x_3$, we find that $k_1 \neq 0$. 

Now $l_2$ divides $p_5 - h_1 p_4$, which can be written as
	\[\of{p_5 - l_1 p_{4}} - x_3 \of{k_1 p_{4}}.\]
It follows that $l_2$ divides both $p_5 - l_1 p_{4}$ and $k_1 p_{4}$ and hence $p_{4}$ and $p_5$ as well. Hence the ideal 
	\[I = (x_1,x_2,x_3,l_2(x_4,x_5))\] 
contains all five of the relations $u_j$, and we once again have a contradiction to the finite-dimensionality of $H^*$.
\end{itemize}

We have derived a contradiction in all cases, so the proof is complete.
\end{proof}

The two propositions in this section imply the following.

\begin{corollary}\label{cor:8}
Let $H^* \cong \Q[x_1,\ldots,x_k]/(u_1,\ldots,u_k)$ be a pure model for a positively elliptic algebra that does not split and has the property that 
	\[|x_{k-1}| + |x_k| \leq 8.\] 
If $H^*$ admits a non-zero derivation with negative degree, then $\fd H^* > 20$.
\end{corollary}

\bigskip\section{The Top-to-Bottom Lemma}\label{sec:Top-to-Bottom}\medskip
In this section, we prove the Top-to-Bottom Lemma and use it to prove Corollary \ref{cor:10}. This is main step in the proof of the Halperin Conjecture when the degrees of the two largest generators satisfy $|x_{k-1}| + |x_k| = 10$.

\begin{lemma}[Top-to-Bottom Lemma]\label{lem:Top-to-Bottom}
Let $H^* \cong \Q[x_1,\ldots,x_k]/(u_1,\ldots,u_k)$ be a positively elliptic algebra that does not split and that satisfies $|u_k| < 3|x_k|$. If there exists a derivation $\delta$ on $H^*$ and $l \geq 1$ such that the map
	\[\delta^l \colon H^{|x_k|} \to H^{|x_1|}\]
exists and is non-zero, then in fact this map has rank at least two.
\end{lemma}

This lemma is reminiscent of the $k - 1$ Lemma, which states that a derivation with negative degree is non-zero only if it has rank at least two.

\begin{proof}
Without loss of generality, we may assume that $|\delta|$ divides $|x_k| - |x_1|$, and we may fix $l \geq 1$ such that $\delta^l$ maps $H^{|x_k|}$ into $H^{|x_1|}$. We may also assume that this map has rank exactly one and change basis, if necessary, so that $\delta^l(x_k) = x_1$ and that $\delta^l(x_i) = 0$ for $i < k$.

Consider the ideal in $H^*$ generated by $x_1,\ldots,x_{k-1}$. Since $H^*$ is finite-dimensional, there exists some relation $u_i$ not in this ideal. Since $|u_i| < 3|x_k|$, we must have
	\[u_i = \lambda x_k^2 + x_k f + g\]
for some non-zero $\lambda \in \Q$ and some $f, g \in \Q[x_1,\ldots,x_{k-1}]$. By scaling $u_i$, we may assume $\lambda = 1$, and then completing the square and replacing $x_k$ by $x_k + \frac 1 2 f$, we may assume that $f = 0$.

We apply $\delta^{2l}$ to this equation. On the left-hand side, we see that $\delta^{2l}(u_i)$ is in the ideal $(u_1,\ldots,u_k)$ and has (minimal) degree $2|x_1|$. In particular, $\delta^{2l}(u_i)$ is a rational linear combination of the $u_i$ with minimal degree. Hence either it is zero or it is $u_1$ after possibly replacing $u_1$ by this linear combination.

On the right-hand side, note that
	\[\delta^{2l}(x_k^2) = \binom{2l}{l} \of{\delta^l x_k}^2= \binom{2l}{l} x_1^2.\] 
If it is the case that $\delta^{2l}(g) = 0$, then both sides of the equation are non-zero and we have $u_1 \in \Q[x_1]$, a contradiction to the assumption that $H^*$ does not split. Hence we may assume that $\delta^{2l}(g) \neq 0$. 

Now $g$ is a polynomial in $x_1,\ldots,x_{k-1}$, so there exists a monomial $x_{i_1}\cdots x_{i_p}$ appearing in $g$ such that \[\delta^{2l}(x_{i_1}\cdots x_{i_p}) \neq 0.\]  Furthermore, by the Leibniz rule, there exists $j_1 + \ldots + j_p = 2l$ such that \[\delta^{j_1}(x_{i_1})\cdots\delta^{j_p}(x_{i_p}) \neq 0.\] Each term in this product is non-zero and hence has degree at least $|x_1|$. Summing, we have
	\[p|x_1| \leq \of{|x_{i_1}| + j_1|\delta|} + \ldots + \of{|x_{i_p}| + j_p|\delta|} = 2|x_k| + 2l |\delta| = 2|x_1|.\]
Hence $p \leq 2$. At the same time, $x_k$ has maximal degree among the generators, so $p = 2$ and equality holds in the estimate above. It follows that some $\delta^l(x_i) \neq 0$ with $x_i \neq x_k$, and this implies a contradiction to our choice of basis at the beginning of the proof.
\end{proof}

Using the Top-to-Bottom Lemma, we can nearly prove the theorem under the condition $|x_{k-1}| + |x_k| = 10$. The exceptional case given in Proposition \ref{pro:10part1} is proved in Proposition \ref{pro:10part2} using another argument.

\begin{proposition}\label{pro:10part1} Let $H^*$ be a positively elliptic algebra that does not split. If there exists a non-zero derivation of negative degree and $|x_{k-1}| + |x_k| = 10$, then either $\fd H^* > 20$ or the degree type is equal to
	\[(2,2,2,4,6; 4,4,6,10,12).\]
\end{proposition}

\begin{proof}
Since $|x_{k-1}|$ and $|x_k|$ are positive, even numbers summing to $10$, and since $|x_{k-1}| \neq 2$ by the Land in Zero and $k-1$ Lemmas, we may assume that $|x_{k-1}| = 4$ and $|x_k| = 6$. In addition, we may assume that
	\[\delta(x_{k-1}) = x_1\]
up to a change in basis. Note also that $k \geq 3$.

First suppose that $|u_k| > 12$. By the condition $SAC(6)$, there is a relation whose degree is properly divisible by six. In particular, $|u_{k-1}| \geq 12$ or $|u_k| \geq 18$, and hence
	\[\sum_{i=k-1}^k \of{|u_i| - |x_i|} = |u_{k-1}| + |u_k| - 10 \geq 16.\]
Note also that
	\[|u_{k-2}| - |x_{k-2}| \geq \max\of{|x_{k-2}|, |x_1| + |x_{k-1}| - |x_{k-2}|}.\]
Since the maximum is at least the average, and since the left-hand side is even, this is at least four. Substituting these estimates into the formula for the formal dimension and applying the Degree Inequality, we have
	\[\fd H^* \geq \sum_{i=k-2}^{k}\of{|u_i| - |x_i|} \geq 20.\]
Since we may assume that $\fd H^* \leq 20$, we have equality everywhere. In particular, $k = 3$ and $|u_1| - |x_1| = 4$. But the $k-1$ Lemma implies that $\delta(x_2) = \lambda x_1 \neq 0$, so Part 2 of the Degree Inequality implies that $|u_1| \geq |x_1| + |x_3|$. This is a contradiction, and we may assume that $|u_k| = 12$.

The Top-to-Bottom Lemma implies that $\delta^2(x_k) = 0$. After replacing $x_k$ by something of the form $x_k - l(x_1,\ldots,x_{k-2})x_{k-1}$, we may assume that
	\[\delta(x_k) = p(x_2,\ldots,x_{k-2}) \neq 0.\]
In particular, $k \neq 3$, since otherwise this expression implies that $\delta(x_3) = 0$, a contradiction to the $k-1$ Lemma. Assume then that $k \geq 4$.

The condition $\delta^2(x_k) = 0$ also means that we can apply the the second part of Lemma \ref{lem:LargeRelations}. Hence $|u_{k-1}| \geq 10$. 

Suppose for a moment that $k \geq 5$. Since $|u_{k-1}| - |x_{k-1}| \geq 6$ and $|u_k| - |x_k| = 6$, we can estimate the formal dimension as above to obtain
	\[\fd H^* \geq (k-3)|x_1| + 4 + 6 + 6 \geq 20.\]
Hence we may assume that equality holds in these estimates. It follows that the degree type is of the form
	\[(2, A_2, A_3, 4, 6; 2+A_2, 2+A_3, 6, 10, 12).\]
But now the bounds $|u_i| \geq 2|x_i|$ for all $i$ imply that $A_3 = 2$ and $A_2 = 2$, so this is the exceptional case given in the conclusion of the proposition. 

We may assume therefore that $k = 4$. In particular,
	\[\delta(x_3) = x_1 \mathrm{~and~} \delta(x_4) = p(x_2),\]
where $p$ is linear if $|x_2| = 4$ and quadratic if $|x_2| = 2$. 

Since $H^*$ is finite-dimensional, not all of the $u_i$ lie in the ideal $I=(x_1,x_2, x_4)$, since otherwise $H^*$ projects onto the infinite-dimensional algebra $\Q[x_1,\ldots,x_4]/I$.  Hence there exists a relation (up to scaling) of the form
	\[u_i = x_{3}^2 + r	\mathrm{~or~} u_i = x_{3}^3 + r.\]
For degree reasons, the structure of $\delta$ implies that $r \in \ker(\delta^2)$ in the first case or $r \in \ker(\delta^3)$ in the second. Applying $\delta^2$ or $\delta^3$, we see that
	\[2 x_1^2 = \delta^2(u_i) \mathrm{~or~} 6 x_1^3 = \delta^3(u_i).\]
Since $\delta$ preserves the ideal $(u_1, \ldots,u_4)$, the right-hand side of each expression lies in this ideal. In the first case, we may perform a change of basis on the degree four $u_i$ to obtain $u_1 = 2x_1^2$. This gives rise to a splitting by the subalgebra generated by $x_1$, a contradiction to the assumptions of the proposition. 

Similarly, the second case gives rise to a contradiction if it is possible to change basis so that some $u_j = 6x_1^3$. Therefore we may assume that
	\[6x_1^3 = \sum l_j u_j\]
where the $l_j$ are linear polynomials in the degree two generators and the $u_j$ are degree four relations. Now if $u_1$ is the only one degree four relation, then $u_1$ is a multiple of $x_1^2$, which is again a contradiction. But then we must have that $|u_1| = |u_2| = 4$, so we have that
	\[|u_2| < 6 = |x_1| + |x_3|.\]
By the Degree Inequality, it follows that $x_1$ and $x_2$ generate a subalgebra of $H^*$ that induces a splitting. This is a contradiction, so the proof is complete.
\end{proof}

We now deal with the exceptional case in the following, the proof of which has a different strategy.

\begin{proposition}\label{pro:10part2}
Let $H^* \cong \Q[x_1,\ldots,x_k]/(u_1,\ldots,u_k)$ be a positively elliptic algebra that does not split. If the degree type is 
	\[(2,2,2,4,6; 4,4,6,10,12),\]
then $H^*$ does not admit a non-zero derivation of negative degree.
\end{proposition}

\begin{proof}
As in the proof of the previous proposition, we may assume that 
	\[\delta(x_4) = x_1 \mathrm{~and~} \delta(x_5) = p(x_2,x_3).\]
Consider the ideal $I = (x_1,x_2,x_3, u_5)$. For degree reasons, $u_j \in I$ for all $j \neq 4$. Since $H^*$ is finite-dimensional, it follows that $u_4 \not\in I$. After scaling $u_4$, if necessary, we have
	\[u_4 = x_4 x_5 + r_4\]
with $r_4 \in I$. 

Note that $r_4$ has degree ten. For degree reasons, it is a polynomial in $\Q^{\geq 3}[x_1,\ldots,x_5]$. Note that $\delta$ preserves this subspace.
Since, in addition, $x_4 \delta(x_5) = x_4 p(x_2,x_3)$ is in this subspace, we have that
	\[\delta(u_4) \in x_1 x_5 + \Q^{\geq 3}[x_1,\ldots,x_5].\]
On the other hand, $\delta(u_4)$ is a degree eight element of the ideal $(u_1,\ldots,u_5)$. For degree reasons, this implies that
	\[\delta(u_4) = \sum_{i=1}^3 h_i u_i\]
with $h_i \in \Q^{\geq 1}[x_1,\ldots,x_5]$. But each $u_j$ is an element of $\Q^{\geq 2}[x_1,\ldots,x_5]$, so $\delta(u_j)$ is as well. Hence this equation shows that $\delta(u_4) \in \Q^{\geq 3}[x_1,\ldots,x_5]$, a contradiction.
\end{proof}

The propositions above imply the following:

\begin{corollary}\label{cor:10}
Let $H^* \cong \Q[x_1,\ldots,x_k]/(u_1,\ldots,u_k)$ be a pure model for a positively elliptic algebra that does not split and that satisfies
	\[|x_{k-1}| + |x_k| = 10.\]
If $H^*$ admits a non-zero derivation with negative degree, then $\fd H^* > 20$.
\end{corollary}

\bigskip\section{Proof of the main theorem}\label{sec:Proof}\medskip

In this section, we finish the proof of the Halperin Conjecture for formal dimensions at most $20$. We are given a positively elliptic algebra
	\[H^* \cong \Q[x_1,\ldots,x_k]/(u_1,\ldots,u_k)\]
as in Theorem \ref{thm:PureModel}, and we assume the existence of a non-zero derivation $\delta$ on $H^*$ of negative degree. We seek a contradiction.

If the formal dimension $\fd H^* = 2$, then Theorem \ref{thm:PureModel} implies that $k = 1$. Hence the Land in Zero Lemma implies that $\delta = 0$, a contradiction. We may therefore inductively assume that $2 < \fd H^* \leq 20$ and that the Halperin Conjecture holds for formal dimensions less than $\fd H^*$.

In particular, we may assume by Markl's theorem that $H^*$ does not split. Hence the Degree Inequality applies to the degrees of the relations $u_i$. Corollaries \ref{cor:8} and \ref{cor:10} also apply, and together they imply that
	\[|x_{k-1}| + |x_k| \geq 12.\]
Putting these facts together, we can finish the proof in all but two exceptional cases.

\begin{proposition}\label{pro:12orLarger-part1} Let $H^* = \Q[x_1,\ldots,x_k]/(u_1,\ldots,u_k)$ be a positively elliptic algebra with no non-trivial subalgebra and $\fd H^* \leq 20$. If there exists a non-zero derivation of negative degree and $|x_{k-1}| + |x_k| \geq 12$, then the degree type is 
	\[(2,4,6,6; 6,8,12,12) \mathrm{~or~} (2,2,6,6; 4,8,12,12).\]
\end{proposition}

\begin{proof}
First suppose that $k = 3$. By the Land in Zero Lemma, $\delta x_1 = 0$. By the $k-1$ Lemma, $\delta x_2$ and $\delta x_3$ are non-zero and, moreover, linearly independent since otherwise we could change basis to ensure $\delta x_2 = 0$. In particular, it follows for degree reasons that $|x_1| < |x_2| < |x_3|$. On one extreme, these degrees could be $2$, $4$, and $6$, but this contradicts the assumption that $|x_{k-1}| + |x_k| \geq 12$. We may assume therefore that $|x_3| \geq 8$. We put this into the formula for the formal dimension in Theorem \ref{thm:PureModel} and we estimate the summands using the Degree Inequality (Lemma \ref{lem:degreeInequality}):
	\[\fd H^*	= \sum_{i=1}^3 \of{|u_i| - |x_i|} \geq |x_3| + \max(|x_1| + |x_3| - |x_2|, |x_2|) + |x_3|.\]
Since the maximum is at least the average, this implies $\fd H^* > 20$, a contradiction.

Next suppose that $k \geq 4$ and $|x_k| \geq 8$. Using the Degree Inequality to estimate $|u_i|$ for $i \leq k-1$ and the estimate $|u_k| \geq 2|x_k|$, we obtain
	\[\fd H^* \geq	\sum_{i=1}^{k-1} \of{|x_1| + |x_{i+1}| - |x_i|} + |x_k| = (k-2)|x_1| + 2|x_k| \geq 20.\]
Hence equality holds everywhere, and we have $k = 4$ and $|x_k| = 8$. Now we repeat this estimate, except we use the bound $|u_i| \geq 2|x_i|$ for the $i = k-1$ term, to get
	\[20 = \fd H^* \geq (k-3)|x_1| + 2|x_{k-1}| + |x_k| \geq 2 + 2|x_{k-1}| + 8.\]
Hence $|x_{k-1}| \leq 4$. By the $k-1$ Lemma, we may assume that $|x_{k-1}| = 4$ and that equality holds in the above estimates. In particular, $|u_1| \leq |u_2| = 6$ and $|u_3| = 10$, so we have a contraction to the $SAC(4,8)$ condition.

Finally, suppose that $k \geq 4$ and $|x_k| \leq 6$. By the assumption in the proposition, we have $|x_{k-1}| = |x_k| = 6$. Estimating as in the previous case, we see that
	\[\fd H^* \geq (k-3)|x_1| + 2|x_{k-1}| + |x_k| \geq 2 + 3(6) = 20.\]
Hence equality holds, and the degree type is of the form
	\[(2,A_2,6,6; 2+A_2, 8, 12, 12)\]
where $A_2 \in \{2,4\}$. These two possibilities correspond to the two degree types in the conclusion of the proposition, so the proof is complete.
\end{proof}

To finish the proof, we only need to consider the two exceptional degree types in Proposition \ref{pro:12orLarger-part1}. Note that, for the first time, the possibility that $\delta$ has degree $-4$ is non-trivial. Indeed, in all previous cases, it is immediate to see that $\delta$ having degree $-4$, $-6$,\ldots implies that $\delta$ is zero on at least $k-1$ generators for degree reasons and hence that $\delta = 0$ by the $k-1$ Lemma.

The first case is simpler and uses ideas similar to previous proofs.

\begin{proposition}\label{pro:12orLarger-part2} If $H^* = \Q[x_1,\ldots,x_k]/(u_1,\ldots,u_k)$ is a positively elliptic algebra with no non-trivial subalgebra and degree type
	\[(2,4,6,6; 6,8,12,12),\]
then $H^*$ does not admit a non-zero derivation with negative degree.
\end{proposition}

\begin{proof}
Suppose first that $\delta(x_2) = x_1$, after possibly rescaling. Applying the Top-to-Bottom Lemma, we see that $\delta(x_i) = \lambda_i x_1^2$ for some $\lambda_i \in \Q$ for $i \in \{3,4\}$. Replacing $x_i$ by $x_i - \lambda_i x_1 x_2$, we find that $x_1,x_3,x_4 \in \ker(\delta)$ in contradiction to the $k-1$ Lemma. Hence we may assume that
	\[\delta(x_1) = 0 \mathrm{~and~} \delta(x_2) = 0.\]

Furthermore, we may assume that $\delta(x_3)$ and $\delta(x_4)$ are linearly independent elements in degree four. In particular, $\delta$ cannot have degree $-4$ (or smaller), so $\delta$ has degree $-2$. After choosing a suitable basis, we may assume that
	\[\delta(x_3) = x_1^2 \mathrm{~and~} \delta(x_4) = x_2.\]

Write
	\[u_j = p_j(x_3,x_4) + r_j\]
for $j \in \{3,4\}$, where $r_j \in (x_1,x_2)$. Note that $\delta^2(r_j) = 0$ for degree reasons, so 
	\[2 p_j(x_1^2, x_2) = \delta^2(u_j) \in (u_1,\ldots,u_4).\]
This is an equation in degree eight, so we have
	\[2 p_j(x_1^2, x_2) = a x_1 u_1 + b u_2\]
for some $a,b \in \Q$. Note that $b = 0$, since otherwise $u_1$ and $u_2$ are polynomials in $x_1$ and $x_2$, which contradicts the assumption that $H^*$ does not have a non-trivial subalgebra. 

Since $b = 0$, we find that $x_1$ divides $p_j(x_1^2, x_2)$ for $j \in \{3,4\}$. This implies that $x_1^2$ divides $p_j(x_1^2, x_2)$, and hence both $p_3(x_3,x_4)$ and $p_4(x_3,x_4)$ are divisible by $x_3$. It follows that 
	\[u_1,\ldots,u_4 \in (x_1, x_2, x_3),\]
which is a contradiction to the finite-dimensionality of $H^*$.
\end{proof}

Finally, we prove the last remaining case. We wish to highlight that the proof in this case differs from all of the previous arguments. Specifically, we do not choose our basis in order to simplify the action of $\delta$, as this does not appear to help us. Rather we choose our basis in order to simplify the form of the relations. 

\begin{proposition}\label{pro:12orLarger-part3} If $H^* = \Q[x_1,\ldots,x_k]/(u_1,\ldots,u_k)$ is a positively elliptic algebra with no non-trivial subalgebra and degree type
	\[(2,2,6,6; 4,8,12,12),\]
then $H^*$ does not admit a non-zero derivation with negative degree.
\end{proposition}

\begin{proof}
Suppose $\delta$ is a non-zero derivation of negative degree, and note that $\delta$ has degree $-2$ or $-4$ by the Land in Zero Lemma. For $j \in \{3,4\}$, write
	\[u_j = p_j(x_3, x_4) + r_j\]
where $r_j \in (x_1, x_2)$. Since $r_j$ has degree $12$ and hence at most one $x_3$ or $x_4$ in each of its monomials, $r_j \in \ker(\delta^2)$.

Note that $p_3$ and $p_4$ are coprime polynomials. Indeed, if $g(x_3,x_4)$ were a non-constant common factor, then all relations $u_j$ are in the ideal $I = (x_1, x_2, g(x_3,x_4))$ and $H^*$ projects onto the infinite-dimensional space $\Q[x_1,\ldots,x_4]/I$, a contradiction.

Since $p_3(x_3,x_4)$ and $p_4(x_3,x_4)$ are coprime, quadratic polynomials, we can choose bases of $\mathrm{span}\{x_3, x_4\}$ and $\mathrm{span}\{u_3,u_4\}$ such that one of the following cases occurs:
	\[(p_3,p_4) = (x_3^2, x_4^2) \mathrm{~or~} (p_3,p_4) = (x_3^2 - \lambda x_4^2, x_3x_4) \mathrm{~with~}\lambda\neq 0.\]
Indeed, up to relabeling and scaling, we may assume that $p_3$ contains an $x_3^2$ term. Completing the square and replacing $x_3$ by something of the form $x_3 + \mu x_4$, we find that $p_3 = x_3^2 - \lambda x_4^2$ for some $\lambda \in \Q$. Subtracting a multiple of $u_3$ from $u_4$ corresponds to subtracting the same multiple of $p_3$ from $p_4$. We can do this so that $p_4 = \mu x_3 x_4 + \nu x_4^2$ for some $\mu, \nu \in\Q$. If $\mu = 0$, the claim follows by rescaling $u_4$ and subtracting a multiple of $u_4$ from $u_3$. If $\mu \neq 0$, we may replace $x_3$ by $\mu x_3 + \nu x_4$. This results in $p_4 = x_3 x_4$. Subtracting now a multiple of $u_4$ from $u_3$ and scaling $u_3$ once more, we find that we are in the second case of the claim. Note here that $\lambda \neq 0$ because $p_3$ and $p_4$ are coprime.

Returning to the expressions for $u_j$, we apply $\delta^2$ to get
	\[2 p_j(\delta x_3, \delta x_4) = \delta^2(u_j) \in (u_1,\ldots,u_4).\]

Suppose first that $\delta$ has degree $-4$, so that $\delta(x_j) \in \mathrm{span}\{x_1,x_2\}$ for $j \in \{3,4\}$. Without loss of generality, we may assume $\delta x_3 = x_1$ and $\delta x_4 = x_2$. Since $p_3$ and $p_4$ are coprime polynomials, so are
	\[\delta u_3 = 2 p_3(x_1, x_2) \hspace{.2in} \mathrm{and} \hspace{.2in} \delta u_4 = 2 p_4(x_1, x_2).\]
But $\delta u_3, \delta u_4 \in \mathrm{span}\{u_1\}$, so we have a contradiction.

Suppose instead that $\delta$ has degree $-2$. Since the expressions for $p_j(\delta x_3, \delta x_4)$ are in degree eight, we have equations of the form
	\[2 p_j(\delta x_3, \delta x_4) = l_j(x_1,x_2) u_1 + k_j u_2\]
for $j \in \{3,4\}$, where the $l_j$ are linear polynomials and the $k_j \in \Q$. 

If some $k_j \neq 0$, we may replace $u_2$ by $l_j(x_1, x_2) u_1 + k_j u_2$ and conclude that $u_1$ and $u_2$ are polynomials in $x_1$ and $x_2$. This implies the existence of non-trivial subalgebra, a contradiction.

We may assume that $k_3 = k_4 = 0$, so that $u_1$ divides $p_3(\delta x_3, \delta x_4)$ and $p_4(\delta x_3, \delta x_4)$. Using the simple formulas for $p_3$ and $p_4$, we see that one of the following happens:
	\begin{enumerate}
	\item $u_1$ divides both $(\delta x_3)^2$ and $(\delta x_4)^2$.
	\item $u_1$ divides both $(\delta x_3)^2 - \lambda (\delta x_4)^2$ and $(\delta x_3)(\delta x_4)$ for some $\lambda \in \Q \setminus\{0\}$.
	\end{enumerate}
In either case, if $u_1$ is irreducible, it follows that $u_1$ divides both $\delta x_3$ and $\delta x_4$. Since all of these elements have degree four, we find that $\delta x_3$ and $\delta x_4$ are linearly dependent. After changing basis once more, we find a contradiction to the $k-1$ Lemma.

Next if $u_1 = l_1 l_2$ is a product of coprime irreducibles, then each irreducible factor divides both $\delta x_3$ and $\delta x_4$ by a similar argument. Moreover, since $l_1$ and $l_2$ are coprime, it follows that $u_1$ divides both of these elements, and we again have a contradiction.

Finally, if neither of these cases occurs, then $u_1 = \lambda l^2$ for some $\lambda \in \Q$ and some linear polynomial $l = l(x_1,x_2)$. But now we can replace $x_1$ or $x_2$ by $l(x_1,x_2)$ and derive the existence of a non-trivial subalgebra of $H^*$, so we again have a contradiction.
\end{proof}


\end{document}